\documentclass{birkau}

\usepackage{amsmath,amssymb,url,tikz}
\usetikzlibrary{arrows}

\numberwithin{equation}{section}

\theoremstyle{plain}
\newtheorem{theorem}{Theorem}[section]
\newtheorem{lemma}[theorem]{Lemma}
\newtheorem{proposition}[theorem]{Proposition}
\newtheorem{corollary}[theorem]{Corollary}

\theoremstyle{definition}

\newcommand{\meet}{\wedge}
\newcommand{\join}{\vee}
\newcommand{\ovr}{\slash}
\newcommand{\under}{\backslash}

\newcommand{\m}{\mathbf}
\renewcommand{\epsilon}{\varepsilon}

\DeclareMathOperator{\op}{op}

\begin{document}


\title[Distributive laws in residuated binars]{Distributive laws in residuated binars}

\author[Wesley Fussner]{Wesley Fussner}
\address{Department of Mathematics\\
University of Denver\\Denver, Colorado, USA}
\email{wesley.fussner@du.edu}

\author[Peter Jipsen]{Peter Jipsen}
\address{Department of Mathematics\\ Chapman University\\Orange, California, USA}
\email{jipsen@chapman.edu}




\begin{abstract}
In residuated binars there are six non-obvious distributivity identities of
$\cdot,\ovr,\under$ over $\wedge, \vee$. We show that in residuated binars with distributive lattice reducts there are some dependencies among these identities; specifically, there are six pairs of identities that imply another one of these identities, and we provide counterexamples to show that no other dependencies exist among these.
\end{abstract}

\maketitle


\section{Introduction}\label{sec:intro}

A \emph{residuated binar} is an algebra ${\m A}=(A,\meet,\join,\cdot,\under,\ovr)$, where $(A,\meet,\join)$ is a lattice, $\cdot$ is a binary operation on $A$, and for all $x,y,z\in A$,
$$x\cdot y\leq z \iff x\leq z\ovr y \iff y\leq x\under z.$$
A \emph{residuated semigroup} is a residuated binar for which $\cdot$ is associative, and a residuated binar possessing an identity element $e$ for $\cdot$ is called \emph{unital}. An expansion of a unital residuated semigroup by a constant designating the identity is called a \emph{residuated lattice} \cite{GJKO}. All of the aforementioned algebras satisfy the distributive laws\footnote{Here and throughout, to reduce the need for parentheses we assume that $\cdot$ has priority over $\under,\ovr$, which in turn have priority over $\meet,\join$. We also write $x\cdot y$ as $xy$.}
\begin{equation}\label{eq:fj}
x (y\join z) = x y\join x z \tag*{$(\cdot\join)$}
\end{equation}
\begin{equation}\label{eq:jf}
(x\join y) z = x z\join y z \tag*{$(\join\cdot)$}
\end{equation}
\begin{equation}\label{eq:lm}
x\under (y\meet z) = x\under y\meet x\under z \tag*{$(\under\meet)$}
\end{equation}
\begin{equation}\label{eq:mr}
(x\meet y)\ovr z = x\ovr z\meet y\ovr z \tag*{$(\meet\ovr)$}
\end{equation}
\begin{equation}\label{eq:rj}
x\ovr (y\join z) = x\ovr y\meet x\ovr z \tag*{$(\ovr\join)$}
\end{equation}
\begin{equation}\label{eq:jl}
(x\join y)\under z = x\under z\meet y\under z \tag*{$(\join\under)$}
\end{equation}
However, in general neither lattice distributivity nor any of the equations
\begin{equation}\label{eq:fm}
x (y\meet z) = x y\meet x z \tag*{$(\cdot\meet)$}
\end{equation}
\begin{equation}\label{eq:mf}
(x\meet y) z = x z\meet y z \tag*{$(\meet\cdot)$}
\end{equation}
\begin{equation}\label{eq:lj}
x\under (y\join z) = x\under y\join x\under z \tag*{$(\under\join)$}
\end{equation}
\begin{equation}\label{eq:jr}
(x\join y)\ovr z = x\ovr z\join y\ovr z \tag*{$(\join\ovr)$}
\end{equation}
\begin{equation}\label{eq:ml}
(x\meet y)\under z = x\under z\join y\under z \tag*{$(\meet\under)$}
\end{equation}
\begin{equation}\label{eq:rm}
x\ovr (y\meet z) = x\ovr y\join x\ovr z \tag*{$(\ovr\meet)$}
\end{equation}
hold in these algebras.

If $t$ is a term in the language of residuated binars (or residuated semigroups), then the \emph{opposite of $t$} is the term $t^{\op}$ defined recursively as follows. For $x$ a variable, set $x^{\op}=x$, and if $s$ and $t$ are terms then set \mbox{$(s\cdot t)^{\op}=t^{\op}\cdot s^{\op}$}, \mbox{$(s\ovr t)^{\op}=t^{\op}\under s^{\op}$}, $(s\under t)^{\op}=t^{\op}\ovr s^{\op}$, $(s\meet t)^{\op}=t^{\op}\meet s^{\op}$, and $(s\join t)^{\op}=t^{\op}\join s^{\op}$ (and $e^{\op}=e$ in the presence of a multiplicative identity $e$). The opposite of an equation $s=t$ is defined by $(s=t)^{\op}=(s^{\op} = t^{\op})$. \emph{Mirror duality} for residuated binars provides that an equation $\epsilon$ holds in the variety of all residuated binars if and only if $\epsilon^{\op}$ does as well. If $\Sigma\cup\{\epsilon\}$ is a set of equations in the language of residuated binars and $\Sigma^{\op} = \{\sigma^{\op} : \sigma\in\Sigma\}$, then $\Sigma\models\epsilon$ holds in the variety of residuated binars if and only if $\Sigma^{\op}\models \epsilon^{\op}$ holds. Observe that \ref{eq:fm}$^{\op}$, \ref{eq:lj}$^{\op}$, and \ref{eq:ml}$^{\op}$ are respectively \ref{eq:mf}, \ref{eq:jr}, and \ref{eq:rm}.

In the presence of a multiplicative identity $e$, left and right prelinearity
\begin{equation}\label{eq:lp}
e\leq x\under y\join y\under x \tag*{$(lp)$}
\end{equation}
\begin{equation}\label{eq:rp}
e\leq x\ovr y\join y\ovr x \tag*{$(rp)$},
\end{equation}
have a connection to the six nontrivial distributive laws given above. In particular, \cite[Proposition 6.10]{BT2003} shows that in residuated lattices satisfying $e$-distributivity
\begin{equation}\label{eq:ed}
(x\join y)\meet e = (x\meet e)\join (y\meet e),\tag*{$(ed)$}
\end{equation}
the equations \ref{eq:lp}, \ref{eq:ml}, and \ref{eq:lj} are pairwise equivalent, as are the equations \ref{eq:rp}, \ref{eq:rm}, and \ref{eq:jr}. Because \ref{eq:lp} and \ref{eq:rp} axiomatize semilinear residuated lattices (i.e., those that are subdirect products of totally-ordered residuated lattices) under appropriate technical hypotheses (see \cite{BT2003}), this provides one explanation of the well-known fact that all six nontrivial distributive laws hold in semilinear residuated lattices. However, a residuated lattice may satisfy all six nontrivial distributive laws even though it is not semilinear (this is the case, e.g., in lattice-ordered groups).

The dependencies among the six nontrivial distributive laws are more complicated in the absence of a multiplicative identity. Sections \ref{sec:implications} and \ref{sec:countermodels} provide a complete description of the dependencies among the nontrivial distributive laws under the hypothesis of lattice distributivity, both for residuated binars and residuated semigroups. Section \ref{sec:additional properties} provides some additional implications among the distributive laws in unital residuated binars, and in the presence of lattice complements. We conclude in Section \ref{sec:open problems} by proposing some open problems.


\section{Implications among the nontrivial distributive laws}\label{sec:implications}

A residuated binar with a distributive lattice reduct may be associated with its \emph{frame}. The frame of a lattice-distributive residuated binar $\m A$ may be obtained by taking the poset of prime filters of the lattice reduct of $\m A$ and endowing it with a ternary relation $R$ defined by
$$R(F,G,H) \iff F\subseteq G\cdot H,$$
where  $F\cdot G = \{xy : x\in G, y\in H\}$ is the complex product of $F$ and $G$. Observe that the ternary relation $R$ on the frame of a residuated binar is antitone in its first coordinate and isotone in its second and third coordinates.

Satisfaction of either of the identities \ref{eq:lj} and \ref{eq:jr} has significant consequences for the frame of a lattice-distributive residuated binar \cite{FP2018}, and the nontrivial distributive laws may be profitably analyzed from the point of view of frames. In fact, for lattice-distributive residuated binars, each of the distributive laws introduced in the previous section may be rendered in terms of an equivalent first-order condition on the corresponding frames by application of ALBA \cite{CP2012}. For instance, the identity \ref{eq:jr} is equivalent to the condition that for all $x,y,p,q,j$,
$$[R(x,j,p) \;\&\; R(y,j,q)]\implies \exists z [x,y\leq z \;\&\; (R(z,j,p) \text{ or } R(z,j,q))].$$
On the other hand, \ref{eq:ml} is equivalent to the condition that for all $x,y,p,q,j$,
$$[R(p,x,j) \;\&\; R(q,y,j)]\implies \exists z [z\leq x,y\;\&\; (R(p,z,j)\text{ or }R(q,z,j))],$$
whereas \ref{eq:lj} is equivalent to the condition that for all $x,y,p,q,j$,
$$[R(x,p,j) \;\&\; R(y,q,j)]\implies \exists z [x,y\leq z\;\&\; (R(z,p,j)\text{ or }R(z,q,j))].$$

\begin{proposition}\label{prop:jr and ml implies lj frame}
Let $\m A$ be a residuated binar with a distributive lattice reduct. If $\m A$ satisfies both \ref{eq:jr} and \ref{eq:ml}, then $\m A$ also satisfies \ref{eq:lj}.
\end{proposition}

\begin{proof}
Suppose that both \ref{eq:jr} and \ref{eq:ml} hold. We use the equivalent frame conditions to verify \ref{eq:lj}, so suppose that $x,y,p,q,j$ are points in the frame of $\m A$ such that $R(x,p,j)$ and $R(y,q,j)$. By the frame condition for \ref{eq:ml} there exists $z'$ with $z'\leq p,q$ and one of $R(x,z',j)$ or $R(y,z',j)$. Suppose first that $R(x,z',j)$ holds. Then $R(x,z',j)$ and $R(y,q,j)$, and by monotonicity and $z'\leq q$ we have $R(x,q,j)$ and $R(y,q,j)$. Using the frame condition for \ref{eq:jr} we obtain $z$ such that $z,y\leq z$ and $R(z,q,j)$. On the other hand, if $R(y,z',j)$ holds then $R(y,z',j)$ and $R(x,p,j)$. Monotonicity and $z'\leq p$ then gives $R(y,p,j)$ and $R(x,p,j)$, and by the frame condition for \ref{eq:jr} there exists $z$ with $x,y\leq z$ and $R(z,p,j)$. In either case, there exists $z$ with $x,y\leq z$ and either $R(z,p,j)$ or $R(z,q,j)$, which completes the proof.
\end{proof}

Other results of this kind may be discovered by appealing to equivalent conditions on frames. However, an entirely algebraic treatment is also possible. The next lemma is an important step in this.

\begin{lemma}\label{lem:four variables}
Each of the following gives a pair of identities that are equivalent in residuated binars.
\begin{enumerate}
\item \ref{eq:fm} and $xz\meet yw\leq (x\join y)(z\meet w)$.
\item \ref{eq:mf} and $xz\meet yw\leq (x\meet y)(z\join w)$.
\item \ref{eq:lj} and $(x\join y)\under (z\join w) \leq x\under z \join y\under w$.
\item \ref{eq:jr} and $(z\join w)\ovr (x\join y) \leq z\ovr x\join w\ovr y$.
\item \ref{eq:ml} and $(x\meet y)\under (z\meet w)\leq x\under z\join y\under w$.
\item \ref{eq:rm} and $(z\meet w)\ovr (x\meet y)\leq z\ovr x\join w\ovr y$.
\end{enumerate}
\end{lemma}

\begin{proof}
We prove (1) and (3); (2) and (4) follow by a symmetric argument, and (5) and (6) follow by a proof similar to (3) and (4).

For (1), note that if $xz\meet yw\leq (x\join y)(z\meet w)$ holds then by instantiating $y=x$ we obtain $xz\meet xw\leq x(z\meet w)$. The reverse inequality follows from the isotonicity of multiplication, so \ref{eq:fm} holds. Conversely, if \ref{eq:fm} holds then we have $xz\meet yw\leq (x\join y)z\meet (x\join y)w = (x\join y)(z\meet w)$.

For (3), taking $y=x$ in the inequality $(x\join y)\under (z\join w)\leq x\under z\join y\under w$ gives \mbox{$x\under (z\join w)\leq x\under z\join x\under w$.} The reverse inequality holds because $\under$ is isotone in its numerator, whence \ref{eq:lj} holds. For the converse, note that \ref{eq:lj} implies $(x\join y)\under (z\join w)=(x\join y)\under z\join (x\join y)\under w\leq x\under z\join y\under w$, where the last step follows because $\under$ is antitone in its denominator.
\end{proof}

\begin{theorem}\label{thm:algebraic implications}
Let $\m A$ be a residuated binar with a distributive lattice reduct. Then:
\begin{enumerate}
\item If $\m A$ satisfies both \ref{eq:jr} and \ref{eq:ml}, then $\m A$ also satisfies \ref{eq:lj}.
\item If $\m A$ satisfies both \ref{eq:lj} and \ref{eq:rm}, then $\m A$ also satisfies \ref{eq:jr}.
\item If $\m A$ satisfies both \ref{eq:fm} and \ref{eq:jr}, then $\m A$ also satisfies \ref{eq:rm}.
\item If $\m A$ satisfies both \ref{eq:mf} and \ref{eq:lj}, then $\m A$ also satisfies \ref{eq:ml}.
\item If $\m A$ satisfies both \ref{eq:ml} and \ref{eq:fm}, then $\m A$ also satisfies \ref{eq:mf}.
\item If $\m A$ satisfies both \ref{eq:rm} and \ref{eq:mf}, then $\m A$ also satisfies \ref{eq:fm}.
\end{enumerate}
\end{theorem}

\begin{proof}
We provide proofs for (1) and (5); (2) and (6) follow by mirror duality. The others follow similarly.

For (1), suppose that $u\leq (x\join y)\under (z\join w)$. Then by residuation we get $x,y\leq x\join y\leq (z\join w)\ovr u$, and by \ref{eq:jr} we have $x\leq z\ovr u\join w\ovr u$ and also $y\leq z\ovr u\join w\ovr u$. Observe that $x= x\meet (z\ovr u \join w\ovr u)$ and $y= y\meet (z\ovr u \join w\ovr u)$, and by distributivity we obtain that $x=x_1\join x_2$ and $y=y_1\join y_2$, where
$$x_1=x\meet (z\ovr u),$$
$$x_2=x\meet (w\ovr u),$$
$$y_1=y\meet (z\ovr u),$$
$$y_2=y\meet (w\ovr u).$$
Note that
$$x_1\leq z\ovr u\implies u\leq x_1\under z\leq (x_1\meet y_2)\under z,$$
$$x_2\leq w\ovr u\implies u\leq x_2\under w\leq (x_2\meet y_1)\under w,$$
$$y_1\leq z\ovr u\implies u\leq y_1\under z\leq (x_2\meet y_1)\under z,$$
$$y_2\leq w\ovr u\implies u\leq y_2\under w\leq (x_1\meet y_2)\under w.$$
Hence we get that $u\leq (x_1\meet y_2)\under (z\meet w)\leq x_1\under z\join y_2\under w$ and likewise \mbox{$u\leq (x_2\meet y_1)\under (z\meet w)\leq x_2\under z\join y_1\under w$}. Also, $u\leq x_1\under z\leq x_1\under z\join y_1\under w$ and $u\leq y_2\under w\leq x_2\under z\join y_2\under w$. This implies that:
\begin{align*}
u &\leq (x_1\under z\join y_2\under w)\meet (x_2\under z\join y_1\under w)\meet (x_1\under z\join y_1\under w)\meet (x_2\under z\join y_2\under w)\\
&= ((x_2\under z\meet x_1\under z)\join y_1\under w)\meet ((x_1\under z\meet x_2\under z)\join y_2\under w)\\
&= (x_1\under z\meet x_2\under z)\join (y_1\under w\meet y_2\under w)\\
&= (x_1\join x_2)\under z \join (y_1\join y_2)\under w\\
&= x\under z\join y\under w.
\end{align*}
This proves that $(x\join y)\under (z\join w)\leq x\under z\join y\under w$, whence (1) follows by Lemma \ref{lem:four variables}(3).

To prove (5), suppose that $(x\meet y)(z\join w)\leq u$. By residuating and \ref{eq:ml}, we obtain $z,w\leq z\join w\leq (x\meet y)\under u = x\under u\join y\under u$. Define
$$z_1=z\meet (x\under u),$$
$$z_2=z\meet (y\under u),$$
$$w_1=w\meet (x\under u),$$
$$w_2=w\meet (y\under u),$$
and note that by the distributivity of the lattice reduct we have $z=z_1\join z_2$ an $w=w_1\join w_2$. This provides
$$z_1\leq x\under u \implies xz_1\leq u,$$
$$z_2\leq y\under u \implies yz_2\leq u,$$
$$w_1\leq x\under u \implies xw_1\leq u,$$
$$w_2\leq y\under u\implies yw_2\leq u,$$
whence from the isotonicity of multiplication and the middle two items above, we obtain that $y(z_2\meet w_1)\leq u$ and $x(z_2\meet w_1)\leq u$. This provides that $(x\join y)(z_2\meet w_1)=x(z_2\meet w_1)\join y(z_2\meet w_1)\leq u$, and from the assumption \ref{eq:fm} and Lemma \ref{lem:four variables}(1) we conclude that $xz_2\meet yw_1\leq u$. Now note that
\begin{align*}
xz\meet yw &= x(z_1\join z_2)\meet y(w_1\join w_2)\\
&= (xz_1\join xz_2)\meet (yw_1\join yw_2)\\
&= (xz_1\meet yw_1)\join (xz_1\meet yw_2)\join (xz_2\meet yw_1)\join (xz_2\meet yw_2)\\
&\leq u,
\end{align*}
where the third equation above follow from lattice distributivity. It follows that $xz\meet yw\leq (x\meet y)(z\join w)$, so \ref{eq:mf} follows by Lemma \ref{lem:four variables}(2). This gives (5).
\end{proof}
The implications articulated in Theorem \ref{thm:algebraic implications} are described by the directed graph in Figure \ref{fig:implications}. Each pair of identities given on the left-hand side (respectively, right-hand side) of the graph jointly imply their common successor on the right-hand side (respectively, left-hand side). Note that these consequences are hidden in the special case of $e$-distributive residuated lattices addressed in \cite{BT2003}, where taken individually \ref{eq:ml} and \ref{eq:lj} are equivalent, as are \ref{eq:jr} and \ref{eq:rm}.

\begin{figure}
\begin{center}
\begin{tikzpicture}[scale=0.4]

\tikzset{vertex/.style = {shape=circle,draw,fill=white,inner sep=1.5pt, minimum size=2em}}
\tikzset{edge/.style = {->,> = latex'}}

\node[vertex] (a) at  (0,8) {$\join\ovr$};
\node[vertex] (b) at  (0,4) {$\meet\under$};
\node[vertex] (c) at  (0,0) {$\cdot\meet$};
\node[vertex] (d) at  (8,8) {$\under\join$};
\node[vertex] (e) at  (8,4) {$\ovr\meet$};
\node[vertex] (f) at  (8,0) {$\meet\cdot$};

\draw[edge] (4,7)  to (d);
\draw[edge] (4,3)  to (e);
\draw[edge] (4,-1) to (f);
\draw[edge] (4,9) to (a);
\draw[edge] (4,5) to (b);
\draw[edge] (4,1) to (c);
\draw (a) to (4,7);
\draw (b) to (4,7);
\draw (a) to (4,3);
\draw (c) to (4,3);
\draw (b) to (4,-1);
\draw (c) to (4,-1);
\draw (d) to (4,9);
\draw (e) to (4,9);
\draw (d) to (4,5);
\draw (f) to (4,5);
\draw (e) to (4,1);
\draw (f) to (4,1);
\end{tikzpicture}
\end{center}

\caption{Dependencies among the nontrivial distributive laws.}
\label{fig:implications}
\end{figure}
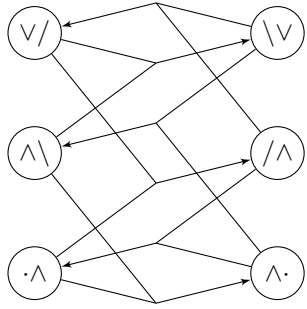

\section{The poset of subvarieties}\label{sec:countermodels}

The class of residuated binars with distributive lattice reducts forms a finitely-based variety $\sf RB$, and the implications announced in Theorem \ref{thm:algebraic implications} entail some inclusions among the subvarieties of $\sf RB$ determined by the nontrivial distributive laws. We will show that these are all of the inclusions among such subvarieties, completely describing the subposet of the subvariety lattice of $\sf RB$ whose elements are axiomatized (modulo the theory of $\sf RB$) by any collection of the nontrivial distributive laws. The same analysis holds for residuated semigroups as well.

\begin{figure}\label{fig:lattice reducts}
\begin{center}
\begin{tikzpicture}

\tikzset{vertex/.style = {shape=circle,draw,fill=white,inner sep=1.5pt}}
\tikzset{edge/.style = {-,> = latex'}}

\node[vertex,label=left:$\top$] (b) at  (0,-1) {};
\node[vertex,label=left:$a$] (c) at  (-1,-2) {};
\node[vertex,label=right:$b$] (d) at  (1,-2) {};
\node[vertex,label=left:$\bot$] (e) at (0,-3) {};

\draw[edge] (b) to (c);
\draw[edge] (b) to (d);
\draw[edge] (c) to (e);
\draw[edge] (d) to (e);
\end{tikzpicture}\hspace{0.25 in}
\begin{tikzpicture}

\tikzset{vertex/.style = {shape=circle,draw,fill=white,inner sep=1.5pt}}
\tikzset{edge/.style = {-,> = latex'}}

\node[vertex,label=left:$\top$] (a) at  (0,0) {};
\node[vertex,label=left:$c$] (b) at  (0,-1) {};
\node[vertex,label=left:$a$] (c) at  (-1,-2) {};
\node[vertex,label=right:$b$] (d) at  (1,-2) {};
\node[vertex,label=left:$\bot$] (e) at (0,-3) {};

\draw[edge] (a) to (b);
\draw[edge] (b) to (c);
\draw[edge] (b) to (d);
\draw[edge] (c) to (e);
\draw[edge] (d) to (e);
\end{tikzpicture}\hspace{0.25 in}
\begin{tikzpicture}

\tikzset{vertex/.style = {shape=circle,draw,fill=white,inner sep=1.5pt}}
\tikzset{edge/.style = {-,> = latex'}}

\node[vertex,label=left:$\top$] (b) at  (0,-1) {};
\node[vertex,label=left:$a$] (c) at  (-1,-2) {};
\node[vertex,label=right:$b$] (d) at  (1,-2) {};
\node[vertex,label=left:$c$] (e) at (0,-3) {};
\node[vertex,label=left:$\bot$] (f) at (0,-4) {};

\draw[edge] (b) to (c);
\draw[edge] (b) to (d);
\draw[edge] (c) to (e);
\draw[edge] (d) to (e);
\draw[edge] (e) to (f);
\end{tikzpicture}
\end{center}

\caption{Labeled Hasse diagrams for the lattice reducts of $\m A_1$, $\m A_2$, $\m A_3$ (left), $\m A_4$, $\m A_5$ (middle) and $\m A_6$ (right).}
\end{figure}
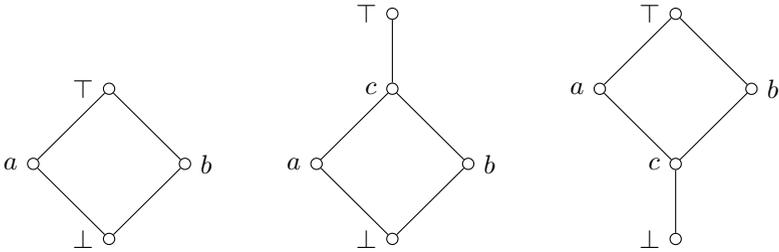

\begin{proposition}\label{prop:no other implications}
Theorem \ref{thm:algebraic implications} gives the only implications among the six nontrivial distributive laws modulo the theory of residuated binars. The same holds for residuated semigroups.
\end{proposition}

\begin{proof}
For each $i\in\{1,2,3,4,5,6\}$ we define a residuated binar $\m A_i$. The lattice reducts of each $\m A_i$ is given in Figure \ref{fig:lattice reducts}. We provide operation tables for $\cdot$ in each $\m A_i$ below; the operation tables for $\under$ and $\ovr$ are uniquely determined by these in each case. For $\m A_1$, $\m A_2$, and $\m A_3$:
$$\begin{array}{c|cccc}
\cdot&\bot&a&b&\top\\\hline
\bot&\bot&\bot&\bot&\bot\\
a&\bot&\bot&\bot&\bot\\
b&\bot&\bot&\top&\top\\
\top&\bot&\bot&\top&\top\\
\end{array}
\qquad
\begin{array}{c|cccc}
\cdot&\bot&a&b&\top\\\hline
\bot&\bot&\bot&\bot&\bot\\
a&\bot&\bot&\bot&\bot\\
b&\bot&a&b&\top\\
\top&\bot&a&b&\top\\
\end{array}
\qquad
\begin{array}{c|cccc}
\cdot&\bot&a&b&\top\\\hline
\bot&\bot&\bot&\bot&\bot\\
a&\bot&\bot&a&a\\
b&\bot&\bot&b&b\\
\top&\bot&\bot&\top&\top\\
\end{array}$$
For $\m A_4$, $\m A_5$, and $\m A_6$:
$$\begin{array}{c|ccccc}
\cdot&\bot&a&b&c&\top\\\hline
\bot&\bot&\bot&\bot&\bot&\bot\\
a&\bot&\top&\bot&\top&\top\\
b&\bot&b&\bot&b&b\\
c&\bot&\top&\bot&\top&\top\\
\top&\bot&\top&\bot&\top&\top\\
\end{array}
\hspace{0.08 in}
\begin{array}{c|ccccc}
\cdot&\bot&a&b&c&\top\\\hline
\bot&\bot&\bot&\bot&\bot&\bot\\
a&\bot&\top&b&\top&\top\\
b&\bot&\bot&\bot&\bot&\bot\\
c&\bot&\top&b&\top&\top\\
\top&\bot&\top&b&\top&\top\\
\end{array}
\hspace{0.08 in}
\begin{array}{c|ccccc}
\cdot&\bot&a&b&c&\top\\\hline
\bot&\bot&\bot&\bot&\bot&\bot\\
a&\bot&a&\bot&\bot&a\\
b&\bot&\bot&b&\bot&b\\
c&\bot&\bot&\bot&\bot&\bot\\
\top&\bot&a&b&\bot&\top\\
\end{array}$$
Direct calculation verifies that:
\begin{itemize}
\item $\m A_1\models$ \ref{eq:rm}, \ref{eq:ml}, \ref{eq:mf}, \ref{eq:fm} and $\m A_1\not\models$ \ref{eq:lj}, \ref{eq:jr}.
\item $\m A_2\models$ \ref{eq:lj}, \ref{eq:ml}, \ref{eq:mf}, \ref{eq:fm} and $\m A_2\not\models$ \ref{eq:jr}, \ref{eq:rm}.
\item $\m A_3\models$ \ref{eq:jr}, \ref{eq:rm}, \ref{eq:mf}, \ref{eq:fm} and $\m A_3\not\models$ \ref{eq:lj}, \ref{eq:ml}.
\item $\m A_4\models$ \ref{eq:jr}, \ref{eq:lj}, \ref{eq:rm}, \ref{eq:fm} and $\m A_4\not\models$ \ref{eq:ml}, \ref{eq:mf}.
\item $\m A_5\models$ \ref{eq:jr}, \ref{eq:lj}, \ref{eq:ml}, \ref{eq:mf} and $\m A_5\not\models $ \ref{eq:rm}, \ref{eq:fm}.
\item $\m A_6\models$ \ref{eq:jr}, \ref{eq:lj}, \ref{eq:rm}, \ref{eq:ml} and $\m A_6\not\models $ \ref{eq:fm}, \ref{eq:mf}.
\end{itemize}
Let $\epsilon\in\{$\ref{eq:jr}, \ref{eq:lj}, \ref{eq:rm}, \ref{eq:ml}, \ref{eq:mf}, \ref{eq:fm}$\}$. Then there exists a unique implication listed in Theorem \ref{thm:algebraic implications} having $\epsilon$ as its consequent. Let $\epsilon_1,\epsilon_2$ be the identities in the antecedent of the aforementioned implication. Then the above countermodels show that if $\epsilon\notin\Sigma\subseteq\{$\ref{eq:jr}, \ref{eq:lj}, \ref{eq:rm}, \ref{eq:ml}, \ref{eq:mf}, \ref{eq:fm}$\}$ and $\epsilon_1\notin\Sigma$ or $\epsilon_2\notin\Sigma$, then $\epsilon$ is not entailed by $\Sigma$.

Note that each $\m A_i$, $i\in\{1,2,3,4,5,6\}$, is an associative residuated binar. The result therefore holds for residuated semigroups as well.
\end{proof}

The left-hand side of Figure \ref{fig:subvariety poset} gives the Hasse diagram of the poset of subvarieties of $\sf RB$ determined by the six nontrivial distributive laws. The coatoms in this diagram are subvarieties axiomatized modulo $\sf RB$ by a single nontrivial distributive law, and the atoms are subvarieties axiomatized by one of the four-element subsets of $\{$\ref{eq:jr}, \ref{eq:lj}, \ref{eq:rm}, \ref{eq:ml}, \ref{eq:mf}, \ref{eq:fm}$\}$ satisfied in one of the models $\m A_i$ given in the proof of Proposition \ref{prop:no other implications}. 
The meets in this diagram correspond to intersection of subvarieties, but in general the joins do not correspond to joins in the lattice of subvarieties.
The same diagram describes the corresponding subvariety poset for residuated semigroups since the models  ${\m A}_i$, $i\in\{1,2,3,4,5,6\}$, are associative.

When $\cdot$ is commutative in a residuated binar $\m A$, the two residuals satisfy $x\under y = y\ovr x$ for all $x,y\in A$ and therefore $\under$ and $\ovr$ coincide. In this event, \ref{eq:lj} is equivalent to \ref{eq:jr}, \ref{eq:ml} is equivalent to \ref{eq:rm}, and \ref{eq:fm} is equivalent to \ref{eq:mf}. The poset of subvarieties axiomatized by the three pairwise independent nontrivial distributive laws is pictured on the right-hand side of Figure \ref{fig:subvariety poset}. The correctness of this diagram can be verified by observing that the models ${\m A}_1$ and ${\m A}_6$ are commutative. Since they are also associative, the same diagram describes the subvariety poset for commutative residuated semigroups. 

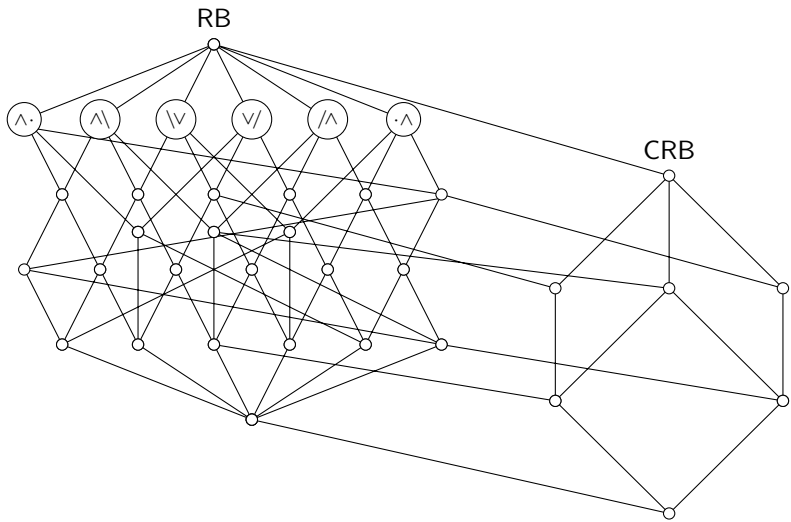
\begin{figure}\label{fig:subvariety poset}
\begin{center}
\begin{tikzpicture}[xscale=.5,  yscale = .5, 
every node/.style={circle, draw, fill=white, inner sep=1.5pt},
t/.style={rectangle,draw=white,fill=white,inner sep=0pt},
g/.style={draw,fill=white,inner sep=1.5pt}
]
\draw(6,0)node{}--(1,2)node{}--(0,4)node{}--(1,6)node{}--(2,8)
--(3,6)node{}--(4,4)node{}--(5,2)node{}--(6,4)node{}--(7,6)node{}
--(8,8)--(9,6)node{}--(10,4)node{}--(11,2)node{}--(6,0)node{}
--(3,2)node{}--(2,4)node{}--(1,6)node{}--(0,8)--(3,5)node{}
--(3,2)node{}--(4,4)node{}--(5,6)node{}--(6,8)--(7,6)node{}
--(8,4)node{}--(9,2)node{}--(10,4)node{}--(11,6)node{}--(10,8)
--(7,5)node{}--(7,2)node{}--(6,0)node{}--(5,2)node{}--(5,5)node{}
--(2,8)node{\scriptsize$\meet\!\under$}--(5,10)node{}--(4,8)--(3,6)node{}--(2,4)node{}
--(1,2)node{}--(7,5)node[g]{}--(4,8)node{\scriptsize$\under\!\join$}--(5,6)node{}--(6,4)node{}
--(7,2)node{}--(8,4)node{}--(9,6)node{}--(10,8)node{\scriptsize$\cdot\meet$}--(5,10)[]node{}
--(8,8)node{\scriptsize$\ovr\!\meet$}--(5,5)node[g]{}--(11,2)node{}(3,5)node{}--(9,2)node{}
--(6,0)node{}(3,5)node[g]{}--(6,8)node{\scriptsize$\join\!\ovr$}--(5,10)node{}--(0,8)node{\scriptsize$\meet\cdot$}
(11,2)--(20,0.5)node{}--(17,-2.5)node{}--(14,0.5)node{}--(14,3.5)node{}
--(17,6.5)node{}--(20,3.5)node{}--(20,0.5)node{}--(17,3.5)node{}--(14,0.5)node{}
(5,5)--(17,3.5)--(17,6.5)--(5,10)
(6,0)--(17,-2.5)
(5,2)--(14,0.5)
(11,6)--(20,3.5)
(5,6)--(14,3.5)
(0,4)--(11,5.9)
(0,7.8)--(11,6)
(0,4)--(11,2)
(5,10.7)node[t]{$\mathsf{RB}$}
(17,7.2)node[t]{$\mathsf{CRB}$};
\end{tikzpicture}
\end{center}

\caption{The poset of subvarieties determined by the nontrivial distributive laws in varieties of residuated binars ${\sf RB}$ and commutative residuated binars ${\sf CRB}$.}
\end{figure}

\section{Identity elements, complements, and prelinearity}\label{sec:additional properties}

We say that a residuated binar is \emph{complemented} if its lattice reduct is complemented, and \emph{Boolean} if its lattice reduct is a Boolean lattice. A unital residuated binar is called \emph{integral} if it satisfies the identity $x\leq e$, where $e$ is the multiplicative identity.\footnote{This usage of \emph{integral} is typical in the study of residuated lattices, and we caution that it conflicts with the common usage in the theory of relation algebras.}
Boolean (unital) residuated binars are called $(u)r$-algebras in \cite{JJR1995}.
Note that if $\cdot$ and $\meet$ coincide in a residuated binar $\m A$, then $\m A$ is term-equivalent to a Brouwerian algebra (i.e., to the bottom-free reduct of a Heyting algebra). If additionally $\m A$ is a Boolean residuated binar, then $\m A$ is (term-equivalent to) a Boolean algebra.

The presence of complements and an identity element in a residuated binar can have a profound impact on whether it satisfies any of the six non-trivial distributive laws, a stark example of which is illustrated by the following lemma.

\begin{lemma}\label{lem:integral implies Boolean}
Let ${\bf A}$ be a unital complemented residuated binar. If ${\bf A}$ is integral, then $\meet$ and $\cdot$ coincide.
\end{lemma}

\begin{proof}
Since $\m A$ is integral, we have $x\cdot y \leq x\meet y$ for all $x,y\in A$. This implies for any $x\in A$ we have that $x\cdot x'\leq x\meet x' = \bot$, where $x'$ is a complement of $x$. On the other hand, since the identity element $e$ is the greatest element of ${\bf A}$ we have also that $x\join x'=e$ for any $x\in A$. Multiplying by $x$ and using \ref{eq:fj}, we obtain $x=x\cdot e = x\cdot(x\join x') = x^2 \join x\cdot x'= x^2\join\bot = x^2$. This gives that ${\bf A}$ is idempotent, whence for any $x,y\in A$, $x\meet y =(x\meet y)\cdot (x\meet y)\leq x\cdot y\leq x\meet y$, i.e., $x\cdot y = x\meet y$.
\end{proof}

Thus the only complemented integral residuated binars are Boolean algebras, which satisfy all six nontrivial distributive laws as well as lattice distributivity. Satisfaction of nontrivial distributive laws also often forces integrality in this setting.

\begin{lemma}\label{lem:integral}
Let $\m A$ be a unital residuated binar. If $e$ has a complement $e'$ and $\m A$ satisfies any one of the distributive laws \ref{eq:fm}, \ref{eq:mf}, \ref{eq:ml}, \ref{eq:rm}, then $\m A$ is integral.
\end{lemma}

\begin{proof}
We prove the result for \ref{eq:fm} and \ref{eq:ml}. The result follows for \ref{eq:mf} and \ref{eq:rm} by a symmetric argument.

First, suppose that $\m A$ satisfies \ref{eq:fm}. Then:
\begin{align*}
e' &= e\cdot e'\\
&\leq \top\cdot e'\\
&= \top\cdot e' \meet \top\\
&= \top\cdot (e'\meet e)\\
&= \top\cdot\bot\\
&= \bot
\end{align*}
where the last equality uses the identity $x\cdot\bot=\bot$, which holds in all residuated binars. It follows that $e=e\join\bot=e\join e'=\top$, hence $e=\top$.

Second, suppose that $\m A$ satisfies \ref{eq:ml}. Note that:
\begin{align*}
\top &= \bot\under\bot\\
&= (e\meet e')\under \bot\\
&= (e\under\bot)\join (e'\under\bot)\\
&= \bot\join (e'\under\bot)\\
&= e'\under\bot,
\end{align*}
giving $\top\leq e'\under\bot$, and by residuation $e'\cdot\top\leq\bot$. As $e\leq\top$ and $\cdot$ is isotone, we get $e'\cdot e\leq e'\cdot\top\leq\bot$. Therefore $e'\leq\bot$, so $e'=\bot$. It follows that $e'=\bot$, yielding again $e=\top$ and completing the proof.
\end{proof}

Combining the previous two lemmas gives the following result.

\begin{corollary}\label{cor:complemented to Boolean}
Let $\m A$ be a complemented unital residuated binar. If $\m A$ satisfies any one of the distributive laws \ref{eq:fm}, \ref{eq:mf}, \ref{eq:ml}, \ref{eq:rm}, then $\m A$ is a Boolean algebra.
\end{corollary}

\begin{proof}
Since $\m A$ is complemented, $e$ has a complement. Lemma \ref{lem:integral} then gives that $\m A$ is integral, and so by Lemma \ref{lem:integral implies Boolean} it follows that $\m A$ is a Boolean algebra.
\end{proof}

\begin{lemma}
Let $\m A$ be a unital Boolean residuated binar. If $\m A$ satisfies any one of the distributive laws \ref{eq:fm}, \ref{eq:mf}, \ref{eq:lj}, \ref{eq:jr}, \ref{eq:ml}, or \ref{eq:rm}, then $\m A$ is integral, and hence is a Boolean algebra.
\end{lemma}

\begin{proof}
Corollary \ref{cor:complemented to Boolean} settles the claim if $\m A$ satisfies any of \ref{eq:fm}, \ref{eq:mf}, \ref{eq:ml}, or \ref{eq:rm}. We therefore prove the claim for $\m A$ satisfying \ref{eq:lj}; it will follow if $\m A$ satisfies \ref{eq:jr} by a symmetric argument. Suppose that $\m A$ satisfies \ref{eq:lj}, and note that $e\leq\top$ implies $\top\under e'\leq e\under e'=e'$. By \ref{eq:lj} and the isotonicity of $\under$ in its numerator, we have:
\begin{align*}
\top &= \top\under\top\\
&= \top\under (e\join e')\\
&= \top\under e\join \top\under e'\\
&\leq \top\under e\join e'.\\
\end{align*}
Hence $\top\under e\join e' = \top$, so $(\top\under e)'\meet e =\bot$. Because $\meet$ has a residual $\to$ in any Boolean residuated binar, we get $e\leq(\top\under e)'\to \bot=(\top\under e)''=\top\under e$. By residuating with respect to $\cdot$, we obtain that $\top\leq e$, and hence $\top = e$.
\end{proof}

\begin{corollary}
In a unital Boolean residuated binar each of the identities \ref{eq:fm}, \ref{eq:mf}, \ref{eq:lj}, \ref{eq:jr}, \ref{eq:ml}, and \ref{eq:lm} is logically-equivalent to the other five.
\end{corollary}

The two prelinearity equations \ref{eq:lp} and \ref{eq:rp} are not expressible in the absence of a multiplicative identity $e$, but for unital residuated binars they enjoy a connection to the nontrivial distributive laws even in the absence of associativity. In particular, inspection of the proofs offered in \cite{BT2003} verifies that in a unital residuated binar satisfying
$$(x\join y)\meet e = (x\meet e)\join (y\meet e),$$
each of \ref{eq:rm} and \ref{eq:jr} implies \ref{eq:lp}, and each of \ref{eq:ml} and \ref{eq:lj} implies \ref{eq:rp}. Without associativity, the converse implications fail. To see this, we may define a five-element residuated binar $\m A_7$ whose lattice reduct is pictured in Figure \ref{fig:prelinearity}. The multiplication $\cdot$ on $\m A_7$ is given in the following table:
$$\begin{array}{c|ccccc}
\cdot&\bot&a&b&e&\top\\\hline
\bot&\bot&\bot&\bot&\bot&\bot\\
a&\bot&a&\bot&a&e\\
b&\bot&\bot&b&b&\top\\
e&\bot&a&b&e&\top\\
\top&\bot&a&\top&\top&\top\\
\end{array}$$
The residuals $\under$ and $\ovr$ are determined uniquely by the above table as well, and with these operations we have $\m A_7\models$ \ref{eq:lp}, \ref{eq:rp}, but each of \ref{eq:rm}, \ref{eq:jr}, \ref{eq:ml}, and \ref{eq:lj} fail in $\m A_7$. Note also that $\m A_7\not\models$ \ref{eq:fm},\ref{eq:mf}, whence prelinearity does not entail either of the latter distributive laws.

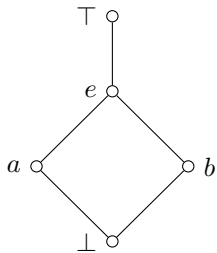
\begin{figure}
\begin{center}
\begin{tikzpicture}

\tikzset{vertex/.style = {shape=circle,draw,fill=white,inner sep=1.5pt}}
\tikzset{edge/.style = {-,> = latex'}}

\node[vertex,label=left:$\top$] (a) at  (0,0) {};
\node[vertex,label=left:$e$] (b) at  (0,-1) {};
\node[vertex,label=left:$a$] (c) at  (-1,-2) {};
\node[vertex,label=right:$b$] (d) at  (1,-2) {};
\node[vertex,label=left:$\bot$] (e) at (0,-3) {};

\draw[edge] (a) to (b);
\draw[edge] (b) to (c);
\draw[edge] (b) to (d);
\draw[edge] (c) to (e);
\draw[edge] (d) to (e);
\end{tikzpicture}
\end{center}

\caption{Hasse diagram for the lattice reduct of $\m A_7$.}
\label{fig:prelinearity}
\end{figure}

\section{Open problems}\label{sec:open problems}

Lattice distributivity is a key ingredient in the known proofs of Theorem \ref{thm:algebraic implications}, whether purely algebraic or by equivalent frame conditions. We do not know whether any of the implications announced hold in all residuated binars (without assuming lattice distributivity), nor do we know whether any of these implications fail in this more general setting.

When present, a multiplicative identity element plays a decisive role in shaping the connection between the nontrivial distributive laws. Known characterizations of when a residuated binar may be embedded in a unital residuated binar crucially involve terms of the form $x\under x$ and $x\ovr x$ (see \cite{B1999,JJR1995}), and we conjecture that conditions involving terms of this form may provide a more satisfying account of the role of a multiplicative identity in this context. In particular, it would be interesting to identify analogues of prelinearity in the non-unital setting and explicate their connection to the nontrivial distributive laws and semilinearity.


\bibliographystyle{plain}

\end{document}